\pdfoutput=1
\documentclass[12pt]{amsart}

\usepackage[
text={440pt,575pt},
headheight=9pt,
centering
]{geometry}

\usepackage{amsmath, amssymb, amsthm, enumerate, bm}
\allowdisplaybreaks 
\usepackage{mathtools}

\usepackage{graphicx}
\usepackage{subcaption}
\usepackage[normalem]{ulem}

\usepackage[all]{xy}

\usepackage{multirow}

\usepackage{mathptmx}

\usepackage{xcolor}
\usepackage{hyperref}
\definecolor{darkblue}{RGB}{0,0,160}
\hypersetup{
	colorlinks,%
	citecolor=darkblue,%
	filecolor=black,%
	linkcolor=darkblue,%
	urlcolor=darkblue
}

\usepackage[capitalise]{cleveref}

\usepackage{comment} 

\makeatletter
\newcommand{\nolisttopbreak}{\vspace{\topsep}\nobreak\@afterheading}
\makeatother

\theoremstyle{definition}
\newtheorem{definition}{Definition}[section]

\newtheorem{theorem}[definition]{Theorem}
\newtheorem{proposition}[definition]{Proposition}
\newtheorem{lemma}[definition]{Lemma}
\newtheorem{corollary}[definition]{Corollary}

\newtheorem*{fact*}{Fact}
\newtheorem{remark}[definition]{Remark}
\newtheorem{example}[definition]{Example}

\newtheorem{problem}[definition]{Problem}

\newcommand{\A}{\mathcal{A}}

\newcommand{\C}{\mathcal{C}}

\newcommand{\Hc}{\mathcal{H}}

\newcommand{\M}{\mathcal{M}}

\newcommand{\NN}{\mathbb{N}}
\newcommand{\Nk}{\mathfrak{N}}

\newcommand{\RR}{\mathbb{R}}
\newcommand{\R}{\mathcal{R}}

\newcommand{\Pk}{\mathfrak{P}}

\newcommand{\ZZ}{\mathbb{Z}}
\newcommand{\Zk}{\mathfrak{Z}}
\newcommand{\QQ}{\mathbb{Q}}

\newcommand\nub{{\boldsymbol 0}}
\newcommand\ab{{\boldsymbol a}}
\newcommand\bb{{\boldsymbol b}}
\newcommand\cb{{\boldsymbol c}}

\newcommand\eb{{\boldsymbol e}}

\newcommand\hb{{\boldsymbol h}}

\newcommand\ub{{\boldsymbol u}}
\newcommand\vb{{\boldsymbol v}}
\newcommand\wb{{\boldsymbol w}}
\newcommand\xb{{\boldsymbol x}}
\newcommand\yb{{\boldsymbol y}}
\newcommand\zb{{\boldsymbol z}}

\DeclareMathOperator{\mndcl}{mn}

\DeclareMathOperator{\Sym}{Sym}

\DeclareMathOperator{\supp}{supp}

\DeclareMathOperator{\cone}{cone}

\DeclareMathOperator{\gp}{gp}

\DeclareMathOperator{\lin}{lin}

\DeclareMathOperator{\rank}{rank}
\DeclareMathOperator{\ind}{ind}

\DeclareMathOperator{\unit}{U}

\newcommand\defas{\coloneqq}

\makeatletter
\@namedef{subjclassname@2020}{%
	\textup{2020} Mathematics Subject Classification}
\makeatother

\title{
	On monoids up to symmetry
 }
\author[D. V. Le]{Dinh Van Le}
\address{Department of Mathematics, FPT University, Hanoi, Vietnam}
\email{dinhlv2@fe.edu.vn}

\author[T. R\"{o}mer]{Tim R\"{o}mer}
\address{Institut f\"ur Mathematik, Universit\"at Osnabr\"uck, 49069 Osnabr\"uck, Germany}
\email{troemer@uos.de}

\author[N. T. Vien]{Nguyen Thi Vien}
\address{Institut f\"ur Mathematik, Universit\"at Osnabr\"uck, 49069 Osnabr\"uck, Germany}
\email{thi.vien.nguyen@uni-osnabrueck.de}
\dedicatory{Dedicated to the memory of Professor J\"{u}rgen Herzog.}
\begin{document}
\begin{abstract}
We study monoids in $\ZZ^{(\NN)}$ that are invariant under the action of the infinite symmetric group $\Sym$. Our main result establishes a local--global principle characterizing equivariant finite generation for arbitrary $\Sym$-invariant monoids, extending earlier results that required additional assumptions. We further analyze local--global phenomena for other fundamental properties, including positivity, normality, seminormality, and simplicity. In addition, we obtain structural results for symmetric monoids, including characterizations of positivity and non-positivity, a description of their groups of units, and explicit formulas for the ranks of local symmetric monoids and stabilizing $\Sym$-invariant chains.
\end{abstract}

\maketitle

\section{Introduction}

This paper is motivated by the goal of developing a theory of polyhedral geometry up to symmetry, with a particular emphasis on monoids. This direction of research originates in \cite{KLR22}, where a general framework was introduced for studying subsets of infinite-dimensional ambient spaces that are invariant under the action of the infinite symmetric group $\Sym$. This framework is connected to developments in a wide range of areas, including algebraic statistics \cite{AHT12,AH07,HS12}, group theory \cite{C67}, representation theory \cite{CEF15,CF13,SS17}, convex optimization \cite{LC23}, and machine learning \cite{LD24,PC25}. Much of this work, including \cite{KLR22}, is inspired by the rapidly growing theory of symmetric ideals in infinite-dimensional polynomial rings. We refer to \cite{D14} for a comprehensive overview and to \cite{C87,D22,HKL18,KLS17,LNNR20,LNNR21,LR24,NR17,NR19} for several significant results in this area.

A key contribution of \cite{KLR22} is the development of a systematic approach to polyhedral objects, such as cones, monoids, and polytopes, in infinite-dimensional settings that are invariant under $\Sym$. The central idea is to relate a \emph{global} $\Sym$-invariant object
$A \subseteq \RR^{(\NN)}$ or $A \subseteq \ZZ^{(\NN)}$
to its associated chain of \emph{local} $\Sym(n)$-invariant objects
$\A = (A_n)_{n \ge 1}$, where $A_n = A \cap \RR^n$ or $A_n = A \cap \ZZ^n$
(precise definitions are recalled in \Cref{sec:preliminaries}.) Properties of the global object $A$ may then be studied through the asymptotic behavior of the local chain $\A$. This perspective leads naturally to the following guiding problem.

\begin{problem}\label{prob:local-global principle}
Characterize fundamental properties of global $\Sym$-invariant objects in terms of their associated chains of local objects.
\end{problem}

Research on \Cref{prob:local-global principle}, commonly referred to as \emph{local--global principles}, has so far focused primarily on equivariant finite generation. In this context, the principle asserts that a global object $A$ is $\Sym$-equivariantly finitely generated if and only if the associated chain $\A$ stabilizes and is eventually finitely generated. This result was first established for nonnegative cones and monoids in \cite{KLR22}, subsequently extended to arbitrary cones and certain classes of monoids in \cite{L25}, and further developed for lattices in \cite{LR24}. Despite this progress, a local--global principle for equivariant finite generation of general $\Sym$-invariant monoids has remained open. Moreover, structural properties beyond finite generation have largely resisted a systematic local--global analysis.

In contrast with symmetric cones, whose equivariant theory is by now well developed, encompassing local--global principles and equivariant analogues of classical results such as the Minkowski--Weyl theorem, Carathéodory’s theorem, and Gordan’s lemma (see \cite{KLR22,L25,LR23}), the study of monoids up to symmetry is still at an early stage. Existing results on $\Sym$-invariant monoids are largely confined to special classes \cite{KLR22,L25}. This limited understanding underscores the need for a more systematic investigation. Such an effort is further motivated by the central role of monoids in algebra and geometry, where they underpin the theory of monoid algebras, toric varieties, and related structures. In the classical finite-dimensional setting, affine monoids and their algebras have been studied extensively, leading to a rich theory with numerous applications; see, e.g., \cite{BG09,BGT02,BH97,BH98,BLR06,G84,GSW76,G89,H70,HRW98,H72,TH86}.

The present work takes a first systematic step toward a theory of monoids up to symmetry. One of our main contributions is a local--global principle characterizing equivariant finite generation for arbitrary $\Sym$-invariant monoids (\Cref{local-global:finite generation}). This result removes auxiliary assumptions imposed in earlier work and thus substantially generalizes the results of \cite{KLR22,L25}. Its proof requires new ideas, as methods developed for cones or nonnegative monoids do not extend to general monoids. In addition, we investigate local--global relationships for other fundamental properties of symmetric monoids, including positivity, normality, seminormality, and simplicity (see \Cref{local-global:positive,local-global:normal-seminormal,local-global:simplicial}), thereby providing further contributions to the study of \Cref{prob:local-global principle}.

Beyond local--global phenomena, we also study intrinsic structural properties of symmetric monoids. In particular, we characterize positivity and non-positivity (\Cref{positivity:global,nonpositivity:global}), describe the groups of units of $\Sym$-invariant monoids (\Cref{unit group:global}), and derive explicit formulas for the ranks of local symmetric monoids and stabilizing $\Sym$-invariant chains (\Cref{rank of monoids} and \Cref{rank of chain}). Together, these results provide a deeper understanding of monoids in infinite-dimensional spaces under symmetry.

The paper is organized as follows. In \Cref{sec:preliminaries}, we review basic notions concerning symmetric groups and symmetric monoids. \Cref{sec:positive and nonpositive} characterizes positive and non-positive symmetric monoids and describes their groups of units in both local and global settings. In \Cref{sec:rank}, we compute the ranks of local symmetric monoids and of stabilizing $\Sym$-invariant chains. Local--global principles for finite generation, normality, seminormality, and simplicity are established in \Cref{sec:local-global}. Finally, \Cref{sec:concluding remarks}  outlines directions for future research on monoids up to symmetry.

\section{Symmetric monoids}\label{sec:preliminaries}

In this section, we recall the notions of symmetric monoids and $\Sym$-invariant chains of monoids introduced in \cite{KLR22}. We begin by describing the ambient group $\ZZ^{(\NN)}$, where the monoids under consideration reside.

Let $\NN=\{1,2,\dots\}$ denote the set of positive integers, and let $\ZZ^\NN$ be the Cartesian product of copies of $\ZZ$ indexed by $\NN$. For any $n\in\NN$, we regard the abelian group $\ZZ^n$ as a subgroup of $\ZZ^\NN$ by identifying each element $(u_1,\dots,u_n)\in\ZZ^n$ with the element
$$(u_1,\dots,u_n,0,0,\dots)\in \ZZ^\NN.$$
This identification yields an increasing chain of subgroups
$$\ZZ \subset \ZZ^2 \subset \cdots \subset \ZZ^n \subset \cdots,$$ whose direct limit is the subgroup
$$\ZZ^{(\NN)}=\bigcup_{n\ge 1} \ZZ^n\subset \ZZ^\NN.$$ 
Thus, $\ZZ^{(\NN)}$ consists precisely of those elements of $\ZZ^\NN$ with finite support. Here, for an element $\ub=(u_1,u_2,\dots)\in\ZZ^\NN$, its \emph{support} is defined as
$$\supp(\ub)=\left\{i \in \NN\mid u_i \neq 0\right\} \subseteq\NN,$$
and the cardinality of $\supp(\ub)$ is called the \emph{support size} of $\ub$. When $\ub\in\ZZ^{(\NN)}$, we denote the \emph{coordinate sum} of $\ub$ by
$$s(\ub)=\sum_{i\in \NN}u_i.$$ 

The abelian group $\ZZ^{(\NN)}$ is free with canonical basis $\{\eb_i \mid i\in\NN\}$, where $\eb_i$ denotes the vector with $1$ in the $i$-th coordinate and $0$ elsewhere. For each $n\in\NN$, the vectors $\eb_1,\dots,\eb_n$ form the standard basis of $\ZZ^n$. We denote by $\epsilon_n$ the all-ones vectors in $\ZZ^n$.

The infinite-dimensional real vector spaces $\RR^\NN$ and $\RR^{(\NN)}$ are defined analogously to $\ZZ^\NN$ and $\ZZ^{(\NN)}$, respectively. For more details, the reader is referred to \cite{L25,LR23}.

\subsection{Symmetric groups}

For each $n\in\NN$, let $\Sym(n)$ denote the symmetric group on the set $[n]:=\{1,\dots,n\}$. Regarding $\Sym(n)$ as the stabilizer of $n+1$ in $\Sym(n+1)$, we obtain an ascending chain of finite symmetric groups
$$\Sym(1) \subset \Sym(2) \subset \cdots \subset \Sym(n) \subset \cdots,$$
whose direct limit is the infinite symmetric group
$$\Sym(\infty)=\bigcup_{n \geq 1} \Sym(n).$$
The group $\Sym(\infty)$ consists of all permutations of $\NN$ that fix all but finitely many elements. For simplicity, we write $\Sym:=\Sym(\infty)$.

There is a natural action of $\Sym$ on $\ZZ^{(\NN)}$ given by permuting coordinates. More precisely, for $\sigma \in \Sym$ and $\ub=\left(u_i\right)_{i \in \NN}=\sum_{i \in \NN} u_i \eb_i \in \ZZ^{(\NN)}$, we define
$$\sigma(\ub)=\left(u_{\sigma^{-1}(1)}, u_{\sigma^{-1}(2)}, \dots\right)=\sum_{i \in \NN} u_{\sigma^{-1}(i)} \eb_i=\sum_{i \in \NN} u_i \eb_{\sigma(i)}.$$
In particular, $\sigma\left(\eb_i\right)=\eb_{\sigma(i)}$ for all $i \in \NN$. For each $n\ge1$, this action restricts to an action of $\Sym(n)$ on $\ZZ^n$.

Given a subset $A\subseteq\ZZ^{(\NN)}$ and a subgroup $\Omega\subseteq\Sym$, let
$$\Omega(A)=\{\sigma(\ub) \mid \sigma \in \Omega, \ub \in A\}$$
denote the $\Omega$-orbit of $A$. We say that $A$ is \emph{$\Omega$-invariant} if $\Omega(A)\subseteq A$. In this paper, we mainly consider the cases $\Omega=\Sym$ and $\Omega=\Sym(n)$ for some $n\ge1$. An ascending chain
$$A_1 \subset A_2 \subset \cdots \subset A_n \subset \cdots,$$
where $A_n \subseteq \ZZ^n$ for all $n \geq 1$, is called \emph{$\Sym$-invariant} if each $A_n$ is $\Sym(n)$-invariant. It is straightforward to verify that the direct limit $A = \bigcup_{n \ge 1} A_n\subseteq \ZZ^{(\NN)}$ of such a chain is $\Sym$-invariant. Conversely, if $A\subseteq\ZZ^{(\NN)}$ is a $\Sym$-invariant subset, then the chain of truncations
$$A \cap \ZZ \subset A \cap \ZZ^2 \subset \cdots \subset A \cap \ZZ^n \subset \cdots$$
is $\Sym$-invariant. Throughout the paper, we refer to a subset or monoid $A \subseteq \ZZ^{(\NN)}$ as a \emph{global} object, while a subset or monoid $A_n \subseteq \ZZ^n$ is referred to as a \emph{local} object.

\subsection{Monoids}

All monoids considered in this paper are submonoids of the abelian group $\ZZ^{(\NN)}$, that is, subsets of $\ZZ^{(\NN)}$ containing the zero element and closed under addition. Let $M\subseteq\ZZ^{(\NN)}$ be a monoid. A subset $A\subseteq M$ is called a \emph{generating set} for $M$ if
$$M=\mndcl(A)\defas\ZZ_{\geq 0}A=\left\{\sum_{i=1}^k m_i \ab_i \mid \ab_i \in A, m_i \in \ZZ_{\geq 0}, k \in \NN\right\}.$$
We say that $M$ is \emph{affine} or \emph{finitely generated} if it admits a finite generating set. Evidently, if $M$ is affine, then $M\subseteq \ZZ^n$ for some $n\in\NN$.

Let $\gp(M)\subseteq \ZZ^{(\NN)}$ denote the \emph{group generated by $M$}. Then $\gp(M)$ is a free abelian group, and
\[
\gp(M)=\ZZ M=\{\ub-\vb\mid \ub,\vb \in M\}.
\]
The \emph{rank} of $M$ is defined to be the rank of the group $\gp(M)$. In other words, 
$$\rank M=\dim_\RR(\RR M),$$ 
where $\RR M$ denotes the $\RR$-vector space spanned by $M$. 

The \emph{normalization} of $M$ is the monoid
$$\widehat{M}=\left\{\ub \in \gp(M) \mid k \ub \in M \text { for some } k \in \NN\right\},$$
and $M$ is called \emph{normal} if $M=\widehat{M}.$ 

It is not hard to show that $\widehat{M}=\cone(M) \cap \gp(M)$, where $\cone(M)\defas \RR_{\ge 0}M$ denotes the \emph{cone} generated by $M$ (see, e.g., \cite[Proposition 2.22]{BG09}). 
This description of $\widehat{M}$ together with the following fundamental result, known as \emph{Gordan’s lemma} (see, e.g., \cite[Lemma~2.9]{BG09} or \cite[Proposition~6.1.2]{BH98}), implies that $\widehat{M}$ is an affine monoid whenever $M$ is affine.

\begin{theorem}
	\label{thm:classical-Gordan}
	Let $G\subseteq\ZZ^{n}$ be a subgroup and $C\subseteq\RR^{n}$ a finitely generated rational cone, i.e., $C=\RR_{\ge 0}A$ for some finite subset $A\subseteq\ZZ^{n}$. Then $C\cap G$ is a normal affine monoid.
\end{theorem}

For an equivariant analogue of Gordan's lemma, we refer the reader to \cite{L25,LR23}.

The monoid $M$ is called \emph{seminormal} if $\ub\in M$ for every $\ub \in \gp(M)$ with $2\ub, 3\ub \in M$. Obviously, every normal monoid is seminormal. On the other hand, there exist seminormal monoids that are not normal (see, e.g., \cite[Example 2.56]{BG09}).

We say that $M$ is \emph{simplicial} if $\cone(M)$ is a simplicial cone, i.e., if $\cone(M)$ is generated by linearly independent elements. 
Note that
\[
\dim \cone(M)=\dim_\RR(\RR M)=\rank M.
\]
Therefore, when $\rank M$ is finite, any generating set of $\cone(M)$ contains at least $\rank M$ many elements, and $M$ is simplicial if and only if $\cone(M)$ has a generating set consisting of exactly $\rank M$ many elements.

The \emph{group of units} of $M$ is defined as
\[
\unit(M)=\{\ub \in M\mid -\ub \in M\}.
\]
If $\unit(M)=\{\nub\}$, then $M$ is called \emph{positive}. This is the case, for instance, when $M \subseteq \ZZ_{\ge 0}^{(\NN)}$.
Any positive affine monoid $M$ has a unique minimal generating set, called the \emph{Hilbert basis} of $M$ and denoted by $\Hc_M$  (see, e.g., \cite[Definition~2.15]{BG09}). 

Now suppose that $M \subseteq \ZZ^{(\NN)}$ is an $\Omega$-invariant monoid, where $\Omega=\Sym$ or $\Omega=\Sym(n)$ for some $n \geq 1$. An \emph{$\Omega$-equivariant generating set} for $M$ is a subset $A\subseteq M$ such that the $\Omega$-orbit of $A$ generates $M$:
$$M=\mndcl(\Omega(A))=\left\{\sum_{i=1}^k m_i \sigma_i\left(\ab_i\right) \mid \ab_i \in A, \sigma_i \in \Omega, m_i \in \ZZ_{\geq 0}, k \in \NN\right\}.$$
If $M$ has a finite $\Omega$-equivariant generating set, then it is said to be \emph{$\Omega$-equivariantly finitely generated} .

To study a global monoid $M \subseteq \ZZ^{(\NN)}$, it is often convenient to consider an \emph{ascending} chain of local monoids $\M=\left(M_n\right)_{n \geq 1}$ with $M_n \subseteq \ZZ^{n}$ such that $M$ is the direct limit of the chain, i.e., $M = \bigcup_{n \ge 1} M_n$. Such a chain $\M$ is called \emph{$\Sym$-invariant} if each $M_n$ is a $\Sym(n)$-invariant monoid in $\ZZ^n$. Evidently, the direct limit of a $\Sym$-invariant chain of local monoids is a $\Sym$-invariant monoid in $\ZZ^{(\NN)}$. Conversely, given a $\Sym$-invariant global monoid $M \subseteq \ZZ^{(\NN)}$, its truncations $(M\cap\ZZ^{n})_{n \ge 1}$ form a $\Sym$-invariant chain.

For any $\Sym$-invariant chain of monoids $\M=\left(M_n\right)_{n \geq 1}$, we always have 
$$\mndcl\left(\Sym(n)\left(M_m\right)\right) \subseteq M_n \quad \text { for all } n \geq m \geq 1.$$
We say that $\M$ \emph{stabilizes} if there exists $r \in \NN$ such that
$$M_n=\mndcl\left(\Sym(n)\left(M_m\right)\right) \quad \text {for all } n \geq m \geq r.$$
The smallest such $r$ is called the \emph{stability index} of $\M$ and is denoted by $\ind(\M)$.

Given a monoidal property $\mathcal{P}$ (e.g., finite generation, positivity, normality, etc.), a chain of monoids $\M = (M_n)_{n \ge 1}$ is said to \textit{satisfy $\mathcal{P}$} (respectively, \textit{eventually satisfy $\mathcal{P}$}) if $M_n$ satisfies $\mathcal{P}$ for all $n \ge 1$ (respectively, for all $n \ge N$ for some $N \in \NN$).

\section{Positive and non-positive symmetric monoids}\label{sec:positive and nonpositive} 

In this section, we establish criteria for determining the positivity and non-positivity of $\Sym$-invariant monoids, thereby extending the classical result for affine monoids (see, e.g., \cite[Proposition~2.16]{BG09}). As a further application, we give an explicit description of the groups of units of $\Sym$-invariant monoids, which yields additional insight into the structure of symmetric monoids. Our approach is to first analyze symmetric monoids in the local setting and then extend the results to the global context.

\subsection{Characterizations of positive and non-positive symmetric monoids}

In \cite[Section 5]{L25}, the first author provided classifications of non-positive symmetric normal monoids $M$ with $\gp(M)=\ZZ^{(\NN)}$ or $\gp(M)=\ZZ^{n}$. We now extend these results by developing characterizations of positivity and non-positivity for general symmetric monoids. In particular, this extension resolves \cite[Problem 5.11]{L25}.

We begin with a criterion for positivity in the local setting. Recall that for $\ub\in \ZZ^{(\NN)}$, the coordinate sum of $\ub$ is denoted by $s(\ub)$.

 \begin{theorem}\label{positivity:local}
     For $n\ge 1$, let $M_n\subseteq \ZZ^{n}$ be a $\Sym(n)$-invariant monoid. Then the following statements are equivalent:
     \begin{enumerate}
         \item 
         $M_n$ is positive;
         \item 
         Either $s(\ub)>0$ for all $\ub\in M_n\setminus\{\nub\}$, or $s(\ub)<0$ for all $\ub\in M_n\setminus\{\nub\}$.
     \end{enumerate}
 \end{theorem}

We illustrate \Cref{positivity:local} with some natural classes of positive symmetric monoids.

 \begin{example}
 Observe that if $M_n\subseteq \ZZ^{n}$ is a $\Sym(n)$-invariant positive monoid such that $s(\ub)>0$ for all $\ub\in M_n\setminus\{\nub\}$, then 
 \[
 -M_n\defas\{\ub\mid-\ub\in M_n\}
 \]
 is also a positive monoid whose nonzero elements have negative coordinate sum. Therefore, it suffices to present examples of the former type.
     \begin{enumerate}
         \item For $n\ge 1$, the $\Sym(n)$-invariant monoid 
         \[ 
         \ZZ^n_{\ge 0}=\mndcl(\Sym(n)(\eb_1))
         \]
         is clearly positive. Moreover, any $\Sym(n)$-invariant submonoid $M_n\subseteq \ZZ^n_{\ge 0}$ is also positive and satisfies $s(\ub)>0$ for all $\ub\in M_n\setminus \{\nub\}$. 
         
         \item Let $n,d \ge 1$ and consider the monoid
         \[ 
         M_n=\{\nub\}\cup\{\ub \in \ZZ^n\mid s(\ub)\in d\NN\}.
         \]
         It is easy to check that $M_n$ is $\Sym(n)$-invariant. Hence, $M_n$ is a $\Sym(n)$-invariant positive monoid satisfying $s(\ub)>0$ for all $\ub\in M_n\setminus\{\nub\}$. 
     \end{enumerate}
 \end{example}
 
\Cref{positivity:local} is evidently equivalent to the following characterization of non-positive symmetric monoids. We therefore restrict our attention to prove the latter.

\begin{proposition}
\label{nonpositivity:local}
    For $n\ge 1$, let $M_n\subseteq \ZZ^{n}$ be a $\Sym(n)$-invariant monoid. Then $M_n$ is non-positive if and only if one of the following conditions holds:
	\begin{enumerate}
		\item  
        There exists $\wb\in M_n\setminus\{\nub\}$ such that $s(\wb)=0$.
		\item 
        There exist $\ub,\vb \in M_n$ such that $s(\ub)>0>s(\vb)$. 
	\end{enumerate}
\end{proposition}

To prove this result, we first record a simple yet useful observation. Recall that $\epsilon_n$ denotes the all-ones vector in $\ZZ^{n}$.

\begin{lemma}
    \label{claim:s(w)epsilon}
    For $n\ge 1$, let $M_n\subseteq \ZZ^{n}$ be a $\Sym(n)$-invariant monoid. Assume that $\wb\in M_n$. 
    Then the following statements hold:
    \begin{enumerate}
        \item 
        $s(\wb)\epsilon_n\in M_n.$
        \item 
        If $s(-\wb)\epsilon_n\in M_n$, then $-\wb\in M_n$.
    \end{enumerate}
    
\end{lemma}

\begin{proof}
   (i) Suppose $\wb=(w_1,w_2,\dots,w_n)$. Consider the following cyclic permutations of $\wb$:
	\begin{align*}
		\wb_1=\wb&=(w_1,w_2,\dots,w_{n-1},w_n),\\
		\wb_2&=(w_2,w_3,\dots,w_n,w_1),\\
		&\vdots\\
		\wb_n&=(w_n,w_1,\dots,w_{n-2},w_{n-1}).
	\end{align*}
Since $M_n$ is $\Sym(n)$-invariant, we have $\wb_i \in M_n$ for all $i\in[n]$. Hence,
\begin{equation}
    \label{eq:cyclic}
    s(\wb)\epsilon_n=\wb_1+\cdots+\wb_n \in M_n.
\end{equation}

(ii) As $s(-\wb)=-s(\wb)$, it follows from \eqref{eq:cyclic} that
\[
-\wb=-\wb_1=\wb_2+\cdots+\wb_n+ s(-\wb)\epsilon_n\in M_n.
\qedhere
\]
\end{proof}

Let us now prove \Cref{nonpositivity:local}.

\begin{proof}[Proof of \Cref{nonpositivity:local}]
 Assume that $M_n$ is non-positive. Then there exists $\ub \in M_n\setminus\{\nub\}$ such that $-\ub \in M_n$. Clearly, either $s(\ub)=s(-\ub)=0$, or $s(\ub)$ and $s(-\ub)$ have opposite signs. Thus, one of the conditions (i) or (ii) holds.

 Conversely, we verify that each of the two conditions implies the non-positivity of $M_n$.

 \emph{Case (i)}: Suppose there exists $\wb\in M_n\setminus\{\nub\}$ such that $s(\wb)=0$. Then $s(-\wb)\epsilon_n=\nub\in M_n$. Applying \Cref{claim:s(w)epsilon}(ii), we get
    $-\wb \in M_n.$
    Hence, $M$ is non-positive.

    \emph{Case (ii)}: Suppose there exist $\ub,\vb \in M_n$ such that $s(\ub)>0>s(\vb)$. Let
    \[
    \wb=s(\ub)\vb - s(\vb)\ub.
    \]
    Then $\wb \in M_n$ and $s(\wb)=0$. If $\wb\ne\nub$, then we are in the situation of Case~(i), and hence $M_n$ is non-positive. Otherwise, if $\wb=\nub$, then 
    $$\xb\defas s(\vb)\ub=s(\ub)\vb  \in M_n.$$ 
    We have $s(\xb)=s(\ub)s(\vb)<0$, so $\xb \ne 0$. Moreover, $- \xb=(-s(\vb))\ub \in M_n$, since $s(\vb)<0$. Therefore, $M_n$ is also non-positive in this case.
\end{proof}

\begin{remark}\label{local positivity: non-equivalent}
	The conditions (i) and (ii) in \Cref{nonpositivity:local} are, in general, not equivalent. Indeed, for $n\ge 1$, consider a monoid $M_n\subseteq \ZZ^n$ that is $\Sym(n)$-equivariantly generated by an element $\ub\in\ZZ^n\setminus\{\nub\}$ with $s(\ub)=0$, i.e.,
    $$M_n=\mndcl(\Sym(n)(\ub)).$$ 
    It is straightforward to verify that every element of $M_n$ then has coordinate sum equal to 0. Thus, (i) does not imply (ii).
    
    Conversely, consider the monoid
    $$M_n=\left\lbrace (a,a,\dots,a) \in \ZZ^n \mid a \in \ZZ\right\rbrace=\mndcl(\epsilon_n,-\epsilon_n).$$ 
    This monoid is $\Sym(n)$-invariant and contains elements $\pm \epsilon_n$ with $s(\epsilon_n)>0>s(-\epsilon_n)$, but it has no non-zero elements whose coordinate sum is 0. Therefore, (ii) does not implies (i) neither.
\end{remark}

Next, we characterize when a $\Sym(n)$-invariant monoid $M_n$ coincides with the entire group $\ZZ^n$. Note that an analogous characterization of $\RR^n$ as a symmetric cone can be found in \cite[Lemma 5.6]{L25}.

\begin{proposition}
 \label{M=Z^n}
     For $n\ge 2$, let $M_n\subseteq \ZZ^{n}$ be a $\Sym(n)$-invariant monoid. Then the following statements are equivalent:
     \begin{enumerate}
         \item 
         $M_n=\ZZ^{n}$;
         \item 
         There exists $\ub,\vb \in M_n$ with $|\supp(\ub)|,|\supp(\vb)| <n$ such that $s(\ub)>0>s(\vb)$ and $\gcd(s(\ub),s(\vb))=1$.
     \end{enumerate}   
\end{proposition}

This result follows readily from the following lemma.

\begin{lemma}\label{dZ^n}
     For $n\ge 2$, let $M_n\subseteq \ZZ^{n}$ be a $\Sym(n)$-invariant monoid. Suppose that there exist $\ub,\vb \in M_n$ such that $|\supp(\ub)|,|\supp(\vb)| <n$ and $s(\ub)>0>s(\vb)$. Then $\pm d\eb_i \in M_n$ for all $i\in[n]$, where $d=\gcd(s(\ub),s(\vb))$.
\end{lemma}

\begin{proof}
By the extended Euclidean algorithm, there exist positive integers $a,b,p,q$ such that 
$$as(\ub)+bs(\vb)=d \quad \text{and} \quad ps(\ub)+qs(\vb)=-d.$$
Set
$$\xb =a\ub+b\vb \quad \text{and} \quad \yb=p\ub +q\vb.$$
Then $\xb,\yb \in M_n$ with $s(\xb)=d$ and $s(\yb)=-d$. By \Cref{claim:s(w)epsilon}(i), we obtain 
$$d\epsilon_n=s(\xb)\epsilon_n\in M_n \quad \text{and} \quad -d\epsilon_n=s(\yb)\epsilon_n\in M_n.$$ 
Since $|\supp(\ub)|,|\supp(\vb)| <n$ and $M_n$ is $\Sym(n)$-invariant, we may permute the coordinates of $\ub,\vb$ so that $\ub,\vb \in M_n \cap \ZZ^{n-1}$. Consequently, $\xb,\yb \in M_n \cap \ZZ^{n-1}$. Note that $M_n \cap \ZZ^{n-1}$ is a $\Sym(n-1)$-invariant submonoid of $\ZZ^{n-1}$. So applying \Cref{claim:s(w)epsilon}(i) to the elements $\xb$ and $\yb$ of this monoid, we deduce that $\pm d\epsilon_{n-1} \in  M_n \cap \ZZ^{n-1}$. Hence,
$$d\eb_n=d\epsilon_n+(-d\epsilon_{n-1})\in M_n \quad \text{and} \quad -d\eb_n=-d\epsilon_n+d\epsilon_{n-1} \in M_n.$$
Finally, since $M_n$ is $\Sym(n)$-invariant, it follows that $\pm d\eb_i \in M_n$ for all $i\in [n]$, as desired.
\end{proof}

As demonstrated by the following example, the condition $|\supp(\ub)|,|\supp(\vb)| <n$ in \Cref{M=Z^n}(ii) is essential and cannot be dropped.

\begin{example}
    Consider the monoid $M_3=\mndcl(\Sym(3)(\xb,\yb))\subseteq \ZZ^{3}$ with $$\xb=(2,0,0) \quad \text{and} \quad \yb=(-3,1,1).$$  
    We have 
    $$s(\xb)=2>0, s(\yb)=-1<0\text{, and }\gcd(s(\xb),s(\yb))=1.$$ 
    However, it is easy to see that any element of $M_3$ with support size less than $3$ must have even coordinate sum. Thus, 
    $$\eb_1\not\in M_3\text{, and so  }M_3 \subsetneq \ZZ^{3}.
    $$
\end{example}

We now extend the criteria for local positivity and non-positivity to the global setting. 

\begin{theorem}\label{positivity:global}
    Let $M\subseteq \ZZ^{(\NN)}$ be a $\Sym$-invariant monoid. Then the following statements are equivalent:
    \begin{enumerate}
        \item 
        $M$ is positive;
        \item 
        Either $s(\ub)>0$ for all $\ub\in M\setminus\{\nub\}$, or $s(\ub)<0$ for all $\ub\in M\setminus\{\nub\}$.
    \end{enumerate}
\end{theorem}

This criterion for the positivity of global symmetric monoids is entirely analogous to the local case given in \Cref{positivity:local}. It is derived from the following characterization of global non-positivity, which differs slightly from its local counterpart in \Cref{nonpositivity:local}.

\begin{proposition}\label{nonpositivity:global}
    Let $M\subseteq \ZZ^{(\NN)}$ be a $\Sym$-invariant monoid. Consider the following statements:
	\begin{enumerate}
    \item 
    $M$ is non-positive.
	\item  
    There exists $\wb\in M\setminus\{\nub\}$ such that $s(\wb)=0$.
	\item 
    There exist $\ub,\vb \in M$ such that $s(\ub)>0>s(\vb)$. 
	\end{enumerate}
    Then (i) and (ii) are equivalent, and (iii) implies both (i) and (ii). 
\end{proposition}

\begin{proof}
Arguing as in the proof of \Cref{nonpositivity:local}, it suffices to justify the implication (i)$\Rightarrow$(ii). Assume that $M$ is non-positive. Then $\pm\ub \in M$ for some $\ub\ne\nub$. Since $\supp(\ub)$ is finite, we may choose $n\gg0$ such that $\ub \in \ZZ^n$. Write $\ub=(u_1,\dots,u_n)$ and define
\[
\wb=(u_1,\dots,u_n,-u_1,\dots,-u_n)\in \ZZ^{(\NN)}.
\]
Then $\wb=\ub+\sigma(-\ub)$ for some $\sigma\in \Sym$. It follows that  $\wb\in M$ since $M$ is $\Sym$-invariant. 
Clearly, $\wb \ne 0$ and $s(\wb)=0$.
This proves (ii), and hence completes the proof.
\end{proof}

\begin{remark}
	A reasoning similar to the one presented in \Cref{local positivity: non-equivalent} shows that, in general, neither condition (i) nor condition (ii) implies condition (iii) in \Cref{nonpositivity:global}.
\end{remark}

As a global counterpart of \Cref{M=Z^n}, we obtain the following characterization of $\ZZ^{(\NN)}$ as a symmetric monoid. An analogous result for cones appears in \cite[Lemma 5.3]{L25}.

\begin{proposition}\label{M=Z^N}
    Let $M\subseteq \ZZ^{(\NN)}$ be a $\Sym$-invariant monoid. Then the following statements are equivalent: 
    \begin{enumerate}
        \item 
        $M=\ZZ^{(\NN)}$;
        \item 
        There exist $\ub,\vb \in M$ such that $s(\ub)>0>s(\vb)$ and $\gcd(s(\ub),s(\vb))=1$.
    \end{enumerate}
\end{proposition}

\begin{proof}
    Only the implication (ii)$\Rightarrow$(i) requires justification. Since the supports of $\ub$ and $\vb$ are finite, we may assume that $\ub,\vb\in \ZZ^{n-1}$ for some sufficiently large $n$. In particular, $|\supp(\ub)|,|\supp(\vb)| <n$ and $\ub,\vb\in M_n\defas M\cap \ZZ^n$. Note that the monoid $M_n$ is $\Sym(n)$-invariant. Applying \Cref{dZ^n} to $M_n$ yields $\pm \eb_{n} \in M_n\subseteq M$. Because $M$ is $\Sym$-invariant, it follows that $\pm \eb_{i} \in M$ for all $i\ge 1$. Hence, $M=\ZZ^{(\NN)}$, as desired.
\end{proof}

\subsection{Groups of units of symmetric monoids}\label{subsec:unit group}

Building on the results established in the previous subsection, we now give explicit descriptions of the groups of units of symmetric monoids. Let us first introduce some notation.
For a subset $A \subseteq \ZZ^{(\NN)}$, define
\begin{align*}
    \Zk(A) &= \{\ub \in A \mid s(\ub)=0\},\\
    \Pk(A) &= \{\ub \in A \mid s(\ub)\ge 0\},\\
    \Nk(A) &= \{\ub \in A \mid s(\ub)\le 0\}.
\end{align*}
When $A=\ZZ^{(\NN)}$ or $A=\ZZ^n$ for $n\ge 2$, it is shown in \cite[Corollaries~5.9 and 5.10]{L25} that
\begin{align*}
    \Zk(A) &= \ZZ \,\Omega(\eb_1-\eb_2),\\
    \Pk(A) &= \ZZ^n_{\ge 0} + \ZZ \,\Omega(\eb_1-\eb_2),\\
    \Nk(A) &= \ZZ^n_{\le 0} + \ZZ \,\Omega(\eb_1-\eb_2),
\end{align*}
where $\Omega=\Sym$ if $A=\ZZ^{(\NN)}$ and $\Omega=\Sym(n)$ if $A=\ZZ^n$ with $n\ge 2$.

Let $M \subseteq \ZZ^{(\NN)}$ be a monoid. Then $\Zk(M)$ is a submonoid of $M$. 
When $M$ is symmetric, its group of units admits an explicit description in terms of this submonoid and $M$ itself.

\begin{proposition}\label{unit group:local}
    Let $n\ge 2$, and let $M_n \subseteq \ZZ^{n}$ be a $\Sym(n)$-invariant monoid. Then:
    \begin{enumerate}
        \item 
        It holds that $\Zk(M_n) \subseteq \unit(M_n) \subseteq M_n.$
        Moreover,
        \[
        \unit(M_n)=
        \begin{cases}
            \Zk(M_n), & \text{if } M_n \subseteq \Pk(\ZZ^n) \text{ or } M_n \subseteq \Nk(\ZZ^n),\\
            M_n,      & \text{otherwise.}
        \end{cases}
        \]
        \item 
        One has $\Zk(M_n)=\unit(M_n)=M_n$ if and only if $M_n \subseteq \Zk(\ZZ^{n}).$
    \end{enumerate}  
\end{proposition}

\begin{proof}
    (i) Let $\ub\in \Zk(M_n)$. Then $s(-\ub)\epsilon_n=-s(\ub)\epsilon_n=\nub$. By \Cref{claim:s(w)epsilon}(ii), it follows that $-\ub\in M_n$, and hence $\ub\in \unit(M_n)$. This shows that $\Zk(M_n) \subseteq \unit(M_n)$.

    If $M_n \subseteq \Pk(\ZZ^n)$, then $s(\ub)\ge0$ for all $\ub\in M$. Since $s(-\ub)=-s(\ub)$, it follows that $s(\ub)=0$ for every $\ub\in \unit(M_n)$, and hence $\unit(M_n) \subseteq \Zk(M_n)$. Therefore, $\unit(M_n)=\Zk(M_n)$. An analogous argument applies when $M_n \subseteq \Nk(\ZZ^n)$.
    
    Now assume that $M_n \nsubseteq \Pk(\ZZ^n)$ and $M_n \nsubseteq \Nk(\ZZ^n)$.
    Since $\Zk(M_n) \subseteq \unit(M_n)$, it suffices to show that
    $M_n \setminus \Zk(M_n) \subseteq \unit(M_n)$.
    Let $\ub \in M_n \setminus \Zk(M_n)$. We may assume
    that $s(\ub)>0$, as the case $s(\ub)<0$ is analogous.
    Because $M_n \nsubseteq \Pk(\ZZ^n)$, there exists $\vb \in M_n$ with $s(\vb)<0$.
    Set $d=\gcd(s(\ub),s(\vb))$.
    By the proof of \Cref{dZ^n}, we have $\pm d\,\epsilon_n \in M_n$.
    Since $s(-\ub)$ is a positive integer multiple of $-d$, it follows that
    $s(-\ub)\,\epsilon_n \in M_n$.
    By \Cref{claim:s(w)epsilon}(ii), this implies that $-\ub \in M_n$.
    Hence, $\ub \in \unit(M_n)$, as wanted.

    (ii) This follows immediately from part~(i) and the observation that 
    \[\Zk(M_n)=M_n \text{ if and only if } M_n \subseteq \Zk(\ZZ^{n}).
    \qedhere
    \]
\end{proof}

A monoid is a group precisely when it coincides with its subgroup of units. Therefore, from \Cref{unit group:local} we immediately obtain the following consequence.

\begin{corollary}\label{cor:sym-local-group}
    Let $n\ge 2$, and let $M_n \subseteq \ZZ^{n}$ be a $\Sym(n)$-invariant monoid.
    Then $M_n$ is a group if and only if one of the following conditions holds:
    \begin{enumerate}
        \item 
        $s(\wb)=0$ for all $\wb \in M_n$.
        \item 
        There exist $\ub,\vb \in M_n$ such that $s(\ub)>0>s(\vb)$.
    \end{enumerate}
\end{corollary}

As global counterparts of \Cref{unit group:local,cor:sym-local-group}, we now describe the groups of units of $\Sym$-invariant monoids and characterize when such monoids are groups. The proofs are entirely analogous to those in the local case and are therefore omitted.

\begin{proposition}\label{unit group:global}
    Let $M\subseteq \ZZ^{(\NN)}$ be a $\Sym$-invariant monoid. Then:
    \begin{enumerate}
        \item 
        It holds that $\Zk(M) \subseteq \unit(M) \subseteq M$. Moreover,
        \[
        \unit (M)
        =\begin{cases}
        \Zk(M) & \text{if } M \subseteq \Pk(\ZZ^{(\NN)}) \text{ or } M \subseteq \Nk(\ZZ^{(\NN)}),\\ 
        M& \text{otherwise.}
        \end{cases}
        \]
        \item 
        One has $\Zk(M)=\unit(M)=M$ if and only if $M \subseteq \Zk(\ZZ^{(\NN)}).$
    \end{enumerate}
\end{proposition}

\begin{corollary}
    \label{cor:sym-global-group}
    Let $M\subseteq \ZZ^{(\NN)}$ be a $\Sym$-invariant monoid. Then $M$ is a group if and only if one of the following conditions holds:
    \begin{enumerate}
		\item  
        $s(\wb)=0$ for all $\wb\in M$.
		\item 
        There exist $\ub,\vb \in M$ such that $s(\ub)>0>s(\vb)$. 
	\end{enumerate}
\end{corollary}


\section{Ranks of symmetric monoids}\label{sec:rank}

In this section, we compute the rank of local symmetric monoids and, as a consequence, derive a formula for the rank of monoids appearing in a stabilizing $\Sym$-invariant chain. This computation is of independent interest and also sheds further light on the structure of finite-dimensional approximations of $\Sym$-invariant monoids.

Recall that the rank of a monoid $M\subseteq \ZZ^{(\NN)}$ can be computed as
\[
    \rank M = \dim_{\RR}(\RR M).
\]
We also recall that, for $n\ge 2$, one has
\[
    \Zk(\ZZ^n)= \{\ub \in \ZZ^n \mid s(\ub)=0\}=\ZZ\,\Sym(n)(\eb_1-\eb_2).
\]

\begin{proposition}\label{rank of monoids}
    Let $n\ge 2$ and let $M_n \subseteq \ZZ^n$ be a nonzero $\Sym(n)$-invariant monoid. Then
    \[
        \rank M_n=
        \begin{cases}
            1   & \text{if } M_n \subseteq \ZZ\,\epsilon_n,\\
            n-1 & \text{if } M_n \subseteq \Zk(\ZZ^n),\\
            n   & \text{otherwise.}
        \end{cases}
    \]
\end{proposition}

\begin{proof}
    Let $L$ be a nonzero $\Sym(n)$-invariant vector subspace of $\RR^n$. By \cite[Lemma 5.7]{L25} and its proof, there are exactly three possibilities for $L$:
    \[
    L=
    \begin{cases}
        \RR\,\epsilon_n & \text{if } L \subseteq \RR\,\epsilon_n,\\
        \Zk(\RR^n)      & \text{if } L \subseteq \Zk(\RR^n),\\
        \RR^n& \text{otherwise.}
    \end{cases}
    \]
    Since $M_n \subseteq \ZZ^n$ is a nonzero $\Sym(n)$-invariant monoid, the $\RR$-vector space $\RR M_n$ is a nonzero $\Sym(n)$-invariant subspace of $\RR^n$. Hence, $\RR M_n$ must be one of the three subspaces listed above. Now the conclusion follows from the facts that
    \[
        \dim_{\RR}(\RR\,\epsilon_n)=1
        \quad\text{and}\quad
        \dim_{\RR}\Zk(\RR^n)
        =\dim_{\RR}\{\ub \in \RR^n \mid s(\ub)=0\}
        = n-1.
        \qedhere
    \]
\end{proof}

\begin{corollary}\label{rank of chain}
    Let $\M=(M_{n})_{n\geq 1}$ be a stabilizing $\Sym$-invariant chain of monoids with stability index $\ind(\M)=r.$ If $M_r\ne\{\nub\}$, then for all $n>r$, we have
    $$\rank M_n=\begin{cases}
n-1& \text{if } M_r \subseteq  \Zk(\ZZ^r),\\
n& \text{otherwise. }
\end{cases}$$
\end{corollary}
\begin{proof}
    Let $\ub \in M_r$ be a nonzero element. Then $\ub \in M_n \setminus \ZZ\epsilon_n$ for $n>r$. Hence
    \[
    M_n \nsubseteq \ZZ\epsilon_n
    \quad\text{for all } n>r.
    \]
    Moreover, it is evident that $$M_r \subseteq  \Zk(\ZZ^r) \text{ if and only if  } M_n \subseteq  \Zk(\ZZ^n)\text{ for all } n>r.$$ 
    Therefore, the claim follows directly from \Cref{rank of monoids}.
\end{proof}


\section{Local--global principles}\label{sec:local-global}

In the study of \Cref{prob:local-global principle}, previous work has focused primarily on characterizing equivariant finite generation of global invariant objects in terms of their associated chains of local objects (see \cite{KLR22,L25,LR23,LR24}). This type of result is commonly referred to as the \emph{local--global principle}. 
In this section, focusing on symmetric monoids, we extend the scope of this principle beyond finite generation to encompass several additional fundamental properties of monoids, including positivity, non-positivity, normality, seminormality, and simplicity. This broader perspective yields a more comprehensive understanding of local--global behavior for monoids in the presence of symmetry.

\subsection{Local–global principle for finite generation}

We begin with a characterization of equivariant finite generation for general $\Sym$-invariant monoids. Extending the cases of nonnegative monoids considered in \cite[Corollary~5.13]{KLR22} and positive affine normal monoids treated in \cite[Corollary~4.5]{L25}, this result resolves \cite[Problem~4.7]{L25}. Analogous statements for lattices and cones can, in particular, be found in \cite[Theorem~3.6]{LR24} and \cite[Theorem~4.1]{L25}, respectively.

	\begin{theorem}\label{local-global:finite generation}
		Let $\M=(M_{n})_{n\geq 1}$ be a $\Sym$-invariant chain of monoids with limit $M=\bigcup_{n\geq 1}M_n.$ Then the following statements are equivalent:
		\begin{enumerate}
			\item $\M$ stabilizes, and is eventually finitely generated;
            \item There exists $p\in \NN$ and a finite subset $G_p \subseteq \ZZ^p$ such that $G_p$ is a $\Sym(n)$-equivariant generating set for $M_n$ for all $n\ge p$;
            \item There exists $q\in \NN$ such that for all $n\ge q$, the following conditions hold:
			\begin{enumerate}
				\item $M \cap \ZZ^{n}= M_{n}$,
				\item $M_{n}$ is finitely generated by elements of support size at most $q$;
			\end{enumerate}
			\item $M$ is $\Sym$-equivariantly finitely generated.
		\end{enumerate}
	\end{theorem}
    
We first illustrate \Cref{local-global:finite generation} with the following example. 

    \begin{example}
        Let $d\in\NN$. Consider the $\Sym$-invariant chain $\M=(M_n)_{n\ge 1}$ defined by
        \[
        M_1=\{\nub\}, \quad \text{and} \quad M_n=\{\ub \in \ZZ^n \mid s(\ub) \in d\ZZ_{\ge 0}\} \text{ for all $n\ge 2$}.
        \] 
        We show that
        \[
        M_n=\mndcl(\Sym(n)(\{d\eb_1\} \cup \{\eb_1-\eb_2\}))
        \quad \text{for all }  n\ge 2.
        \]
        Indeed, the inclusion $\supseteq$ is immediate. Conversely, for $\ub=(u_1,\dots,u_n)\in M_n$, we have
        \[
        \ub=s(\ub)\eb_1+\sum_{i=2}^n u_i(\eb_i-\eb_1) \in \mndcl(\Sym(n)(\{d\eb_1\} \cup \{\eb_1-\eb_2\})).
        \]
        Thus, $\M$ stabilizes with $\ind(\M)=2$, and
        $G_2= \{d\eb_1\}\cup\{\eb_1-\eb_2\} \subseteq \ZZ^2$
        is a $\Sym(n)$-equivariant generating set for $M_n$ for all $n\ge 2$. It follows that $G_2$ is a $\Sym$-equivariant generating set for the limit of the chain:
        \[
        M=\bigcup_{n\geq 1}M_n=\{\ub \in \ZZ^{(\NN)} \mid s(\ub) \in d\ZZ_{\ge 0}\}
        =\mndcl(\Sym(G_2)).
        \]
        Observe that $M \cap \ZZ^{n}= M_{n}$ for all $n\ge 2$.
    \end{example}

    While \Cref{local-global:finite generation} extends the nonnegative case considered in \cite[Corollary 5.13]{KLR22}, its proof requires a completely new idea. In fact, a key step in the proof is to establish the equality $M \cap \ZZ^{n}= M_{n}$ for $n\gg0$. In the nonnegative setting, this follows from the elementary observation that if
    $$\ub = \sum_{i=1}^{k}a_i\vb_i \in M\cap \ZZ^n$$
    with $a_i \in \NN$ and $\vb_i \in M$, then necessarily $\vb_i \in M\cap \ZZ^n$ for all $i\in [k]$. This argument crucially relies on nonnegativity and therefore does not extend to the general case. 
    
    For general $\Sym$-invariant chains of cones $\C = (C_n)_{n\ge 1}$, a different approach was used in \cite[Theorem~4.1]{L25} to prove that $C \cap \RR^{n} = C_n$ for $n \gg 0$. It exploits the fact that $C_n=C_n^{**}$, where $C_n^{**}$ denotes the double dual cone of $C_n$. However, this strategy is also not applicable in the present setting, since the double dual monoid $M_n^{**}$ is always normal (see \cite{LRV}), and hence one generally has $M_n\ne M_n^{**}$.

    To prove \Cref{local-global:finite generation}, our main idea is to introduce the monoid
    $$I=\left\lbrace (a_1,\dots, a_N) \in \ZZ_{\geq 0}^N \mid a_1g_1+\cdots+a_Ng_N=0\right\rbrace,$$ where $g_1,g_2,\dots,g_N$ are the coordinates of all vectors in a finite equivariant generating set $G_p\subseteq \ZZ^p$ of the limit monoid $M$. We will show that $I$ has a finite Hilbert basis $\Hc_I$ and that $M\cap\ZZ^n = M_n$ for all $n\ge p\alpha$, where 
    $$\alpha=\max \{1,\, s(\ab)\mid \ab \in \Hc_I\}.$$
    Before turning to the details, let us begin with the following key lemma.

\begin{lemma}\label{replace_coordinate}
	Let $p,\, s\in\NN$ and let $G \subseteq \ZZ^p$ be a nonempty set. For $n\ge p$, let $G_n\subseteq \ZZ^n$ be the $\Sym(n)$-orbit of $G$ and $M_n\subseteq \ZZ^n$ the monoid generated by $G_n$, i.e.,
    \[
    G_n=\Sym(n)(G)
    \quad\text{and}\quad
    M_n=\mndcl(G_n).
    \]
    Suppose that $\vb \in M_{n+1} \cap \ZZ^n$ for some $n \ge ps$. If $\vb$ can be written as the sum of at most $s$ elements of $G_{n+1}$, then $\vb \in M_n$.
\end{lemma}

\begin{proof}
	Since $G_n=\Sym(n)(G)$, every element of $G_n$ has support size at most $p$. Moreover, $G_{n+1}\cap\ZZ^n=G_n$.
    Write 
    $$\vb=\wb_1+\cdots+\wb_k\text{, where } \wb_i=(w_{i,1},\dots,w_{i,n+1})\in G_{n+1} \text{ and } k\le s.$$ 
    
    We first claim that there exists $j\in [n+1]$ such that $w_{i,j}=0$ for all $i\in [k]$. Suppose, to the contrary, that for every $j \in [n+1]$ there exists some $i \in [k]$ with $w_{i,j} \neq 0$.
    Consider the vector 
    $$\wb=(\wb_1,\dots,\wb_k) \in \ZZ^{(n+1)k}.$$ 
    Since each $\wb_i \in G_{n+1}$ has support size at most $p$, we have
    $$|\supp(\wb)|=\sum_{i=1}^{k}|\supp(\wb_i)|\le pk.$$ 
    On the other hand, we may decompose $\wb$ as
    \[
        \wb = \ub_1 + \cdots + \ub_{n+1},
    \]
    where for each $j \in [n+1]$ the vector
    \[
        \ub_j = (0,\dots,0,w_{1,j},0,\dots,0,w_{k,j},0,\dots,0)
        \in \ZZ^{(n+1)k}
    \]
    is obtained from $\wb$ by retaining only the coordinates $w_{i,j}$ for $i \in [k]$ and replacing all others by $0$.
    By the assumption above, for each $j \in [n+1]$ there exists some $i \in [k]$ such that $w_{i,j} \neq 0$, and therefore $|\supp(\ub_j)| \ge 1$.
    Consequently,
    \[
        |\supp(\wb)|
        = \sum_{j=1}^{n+1} |\supp(\ub_j)|
        \ge n+1
        > ps
        \ge pk,
    \]
    which contradicts the inequality $|\supp(\wb)| \le pk$.
    This establishes the claim.

    Now write $\vb=(v_1,\dots,v_{n+1})$. Then $v_{n+1}=0$ because $\vb\in\ZZ^n$. By the claim, there exists $j \in [n+1]$ such that $w_{i,j} = 0$
    for all $i \in [k]$. Hence,
    \[
        v_j = \sum_{i=1}^k w_{i,j} = 0=v_{n+1}.
    \]
    It follows that
    \[
    \vb=\wb_1+\cdots+\wb_k
    = \wb_1' + \cdots + \wb_k',
    \]
    where for each $i \in [k]$, the vector $\wb_i'\in G_{n+1}$ is obtained from $\wb_i$ by swapping the $j$-th and $(n+1)$-st coordinates. Since the $(n+1)$-st coordinate of $\wb_i'$ is $0$, we have $\wb_i'\in\ZZ^n$, and hence $\wb_i'\in G_{n+1}\cap\ZZ^n=G_n$ for all $i \in [k]$.
    Therefore,
    \[
        \vb = \wb_1' + \cdots + \wb_k'\in \mndcl(G_n)= M_n,
    \]
    as claimed.
\end{proof}

We are now ready to prove \Cref{local-global:finite generation}.

	\begin{proof}[Proof of \Cref{local-global:finite generation}]
    While the implications (iii)$\Rightarrow$(ii)$\Rightarrow$(i)$\Rightarrow$(iv) 
    follow analogously to the proof of \cite[Theorem~4.1]{L25}, the implication
    (iv)$\Rightarrow$(iii) is the critical step.
    For the convenience of the reader, we present a complete proof of all implications.   
    \smallskip

		(i)$\Rightarrow$(iv): Assume that $\M$ stabilizes with $\ind(\M)=r$. Choose $p \ge r$ such that $M_p$ has a finite generating set $G_p$. Since $\ind(\M)=r$, we have $$M_n=\mndcl(\Sym(n)(G_p)) \subseteq \mndcl(\Sym(G_p)) \quad \text{for all } n\ge p.$$
		It follows that $$M=\bigcup_{n\geq 1}M_n=\bigcup_{n\geq p}M_n\subseteq\bigcup_{n\geq p}\mndcl(\Sym(G_p))=\mndcl(\Sym(G_p)).$$
		On the other hand, $\mndcl(\Sym(G_p)) \subseteq M$ since $G_p \subseteq M$ and $M$ is $\Sym$-invariant. Therefore, $M=\mndcl(\Sym(G_p))$, and $M$ is $\Sym$-equivariantly finitely generated.
        \smallskip
		
		(iv)$\Rightarrow$(iii): Let $G$ be a finite $\Sym$-equivariant generating set for $M$. Then $G\subseteq M_p$ for some $p\ge 1$. Let $g_1,\dots,g_N$ denote all coordinates of the vectors in $G$. Note that $N\le p|G|$. Consider the monoid
        $$I=\left\lbrace (a_1,\dots, a_N) \in \ZZ_{\geq 0}^N \mid a_1g_1+\cdots+a_Ng_N=0\right\rbrace.$$
        We have $I=C\cap \ZZ^N$, where
        $$C=\left\lbrace (a_1,\dots, a_N) \in \RR_{\geq 0}^N \mid a_1g_1+\cdots+a_Ng_N=0\right\rbrace.$$
        Since $C$ is the intersection of finitely many rational linear halfspaces, it is a rational cone (see \cite[Section~1.G]{BG09}). By Gordan’s lemma (\Cref{thm:classical-Gordan}), $I$ is an affine monoid. Moreover, $I$ is positive because $I\subseteq \ZZ_{\geq 0}^N.$
        Therefore, $I$ has a finite Hilbert basis $\Hc_I$.
        Define
        $$\alpha = \text{max}\left\lbrace 1,\, s(\ab) \mid \ab \in \Hc_I\right\rbrace
        \quad\text{and}\quad
        q=p\alpha.$$
        Then $q\ge p$. To prove (iii), it suffices to show that
        \[
        M_n= M\cap \ZZ^n = \mndcl(\Sym(n)(G))
        \quad\text{for all } n\ge q.
        \]
        Set $G_n=\Sym(n)(G)$ and $M_n'=\mndcl(G_n)$. We first claim that
        \[
        M_{n+1}'\cap \ZZ^n=M_n'
        \quad\text{for all } n\ge q.
        \]
        The inclusion $\supseteq$ is obvious. For the reverse inclusion, let $\ub=(u_1,\dots,u_{n+1}) \in M_{n+1}' \cap \ZZ^n$. Then $u_{n+1}=0$, and we may write
        \begin{equation}
            \label{eq:u}
            \ub=\sum_{i=1}^k\vb_i,
        \quad
        \vb_i\in G_{n+1}.
        \end{equation}
        Note that the $(n+1)$-st coordinate of each $\vb_i$ is one of the elements $g_1,\dots,g_N$. For $j\in [N]$, let $b_j$ be the number of summands $\vb_i$ whose $(n+1)$-st coordinate equals $g_j$. Equation~\eqref{eq:u} then yields
         \begin{equation}
             \label{eq:bg}
             b_1g_1+\cdots+b_Ng_N=u_{n+1}=0.
         \end{equation}
         This implies that $\bb\defas(b_1,\dots,b_N) \in I$. Hence, there exist $\cb_1,\dots,\cb_t\in \Hc_I$ such that
         $$\bb=\cb_1+\cdots+\cb_t.$$ 
         Suppose $\cb_l=(c_{l,1},\dots,c_{l,N})$ for all $l \in [t]$. Then \eqref{eq:bg} can be rewritten  as 
        \begin{equation}\label{eq2:u_n+1}
            (c_{1,1}g_1+\cdots+c_{1,N}g_N)+\cdots+(c_{t,1}g_1+\cdots+c_{t,N}g_N)=u_{n+1}=0.
        \end{equation}
        Using this decomposition, we partition the summands $\vb_i$ in \eqref{eq:u} into $t$ groups as follows.
        For each $l\in[t]$, the $l$-th group consists of $c_{l,1}$ elements with $(n+1)$-st coordinate $g_1$, $c_{l,2}$ elements with $(n+1)$-st coordinate $g_2$, and so on. Let $\wb_l$ denote the sum of the elements in the $l$-th group. Since $\cb_l=(c_{l,1},\dots,c_{l,N}) \in \Hc_I\subseteq I$, the $(n+1)$-st coordinate of $\wb_l$ is 
        \[
        c_{l,1}g_1+\cdots+c_{l,N}g_N=0.
        \]
        Thus, $\wb_l\in M_{n+1}'\cap\ZZ^n$. Moreover, $\wb_l$ is a sum of $c_{l,1}+\cdots+c_{l,N}=s(\cb_l)\le\alpha$ elements of $G_{n+1}$. Since $n\ge q=p\alpha$, \Cref{replace_coordinate} implies $\wb_l\in M_n'$ for all $l\in[t]$. Therefore,
        \[
        \ub=\wb_1+\cdots+\wb_t\in M_n'.
        \]
        This shows that $M_{n+1}'\cap \ZZ^n=M_n'$ for all $n\ge q$, as claimed. 

        Now let $m>n\ge q$.
        Iterating the above equality, we obtain
        \begin{align*}
        M_m'\cap\ZZ^n
        &= (M_m'\cap\ZZ^{m-1})\cap\ZZ^n
         = M_{m-1}'\cap\ZZ^n \\
        &=\cdots
         = M_{n+1}'\cap\ZZ^n
         = M_n'.
        \end{align*}
		Since 
        $$M=\mndcl (\Sym(G))=\bigcup_{m>n}\mndcl (\Sym(m)(G))=\bigcup_{m>n}M_m',$$
        it follows that
		\begin{align*}
			M \cap \ZZ^n&=\bigcup_{m>n}M_m' \cap \ZZ^n=\bigcup_{m> n}M_n'=M_n' \quad \text{for all } n\ge q.
		\end{align*}
		On the other hand, we always have
		$$ M_n' \subseteq M_n \subseteq M\cap \ZZ^n 
        \quad \text{for all } n\ge q,$$ 
        where the first inclusion follows from $G\subseteq M_n$ and the
        $\Sym(n)$-invariance of $M_n$.
        Hence equality holds throughout, and (iii) is proved. 
        \smallskip

		(iii)$\Rightarrow$(ii): We show that one may choose $p=q$. Indeed, for $n\ge q$, let $G_n$ be a generating set of $M_n$ whose elements have support size at most $q$. Then for each $\ub\in G_n$, there exists $\sigma \in \Sym(n)$ such that $\ub'=\sigma(\ub)\in \ZZ^q$. Since $\ub=\sigma^{-1}(\ub')$ and $\ub'\in M\cap \ZZ^q =M_q$, we obtain $G_n \subseteq \Sym(n)(M_q)$. Consequently,
        \[
        M_n=\mndcl(G_n)
        \subseteq \mndcl(\Sym(n)(M_q))
        \subseteq M_n,
        \]
        and equality holds throughout.
        Thus
        \[
        M_n=\mndcl(\Sym(n)(M_q))=\mndcl(\Sym(n)(G_q))
        \quad\text{for all } n\ge q,
        \]
        as desired. 
        \smallskip

        Finally, the implication (ii)$\Rightarrow$(i) is immediate. The proof is complete.
        \end{proof}

To illustrate the idea used in the proof of \Cref{local-global:finite generation}, we consider the following example.

    \begin{example}
        Let $M$ be the $\Sym$-equivariant monoid generated by 
        $$G=\{(1,-2,3), (1,4,-2)\} \subseteq \ZZ^3.$$
        The set of all coordinates appearing in vectors of $G$ is $\{1,-2,3,4\}$. Consider the monoid
        $$I=\left\lbrace (a_1,a_2,a_3,a_4) \in \ZZ_{\geq 0}^4 \mid a_1-2a_2+3a_3+4a_4=0\right\rbrace.$$ 
        Using 
        Macaulay2 \cite{GS}, we find the following Hilbert basis for $I$:
        $$\Hc_I=\{(0,2,0,1),\ (0,3,2,0),\ (1,2,1,0),\ (2,1,0,0)\}.$$ Thus,
        \[
        \alpha
        =\max\{1,\, s(\ab)\mid \ab\in\Hc_I\}
        =5.
        \]
        By \Cref{local-global:finite generation}, any $\Sym$-invariant chain of monoids $\M=(M_{n})_{n\geq 1}$ with limit monoid $M$ stabilizes and is eventually finitely generated.
        Moreover, if $p$ is the smallest index such that $G\subseteq M_p$, then the proof of \Cref{local-global:finite generation} shows that 
        $$M_n=M\cap \ZZ^n=\mndcl(\Sym(n)(G))=\mndcl(\Sym(n)((1,-2,3), (1,4,-2)))$$
        for all $n\ge p\alpha=5p.$
    \end{example}
    
\subsection{Local--global principles for positivity and non-positivity}
\label{subsec:local-global on positivity}

Building on the results of \Cref{sec:positive and nonpositive} concerning positive and non-positive symmetric monoids in both local and global settings, we now examine how the positivity or non-positivity of a stabilizing $\Sym$-invariant chain of monoids relates to that of its limit monoid.

\begin{proposition}\label{local-global:positive}
Let $\M=(M_n)_{n\ge 1}$ be a stabilizing $\Sym$-invariant chain of monoids with limit
$
M=\bigcup_{n\ge 1} M_n
$
and stability index $\ind(\M)=r$. Then the following statements are equivalent:
\begin{enumerate}
    \item $M_n$ is positive for all $n\ge 1$;
    \item $M_r$ is positive;
    \item $M$ is positive.
\end{enumerate}
\end{proposition}

\begin{proof}
The implication \textup{(i)}$\Rightarrow$\textup{(ii)} is immediate.
The implication \textup{(iii)}$\Rightarrow$\textup{(i)} follows since each $M_n$ is a submonoid of $M$.
It remains to prove \textup{(ii)}$\Rightarrow$\textup{(iii)}.  
By \Cref{positivity:local}, $s(\ub)$ has the same sign for all $\ub\in M_r\setminus\{\nub\}$. Since
$
M=\mndcl(\Sym(M_r)),
$
the same holds for all $\vb\in M\setminus\{\nub\}$. Therefore, by \Cref{positivity:global}, $M$ is positive.
\end{proof}

\begin{proposition}\label{local-global:nonpositive}
Let $\M=(M_n)_{n\ge 1}$ be a stabilizing $\Sym$-invariant chain of monoids with limit
$
M=\bigcup_{n\ge 1} M_n
$
and stability index $\ind(\M)=r$. Then the following statements are equivalent:
\begin{enumerate}
    \item $M_n$ is non-positive for all $n\ge r$;
    \item $M_r$ is non-positive;
    \item $M$ is non-positive.
\end{enumerate}
\end{proposition}

\begin{proof}
The implication \textup{(i)}$\Rightarrow$\textup{(ii)} is immediate. The implication \textup{(ii)}$\Rightarrow$\textup{(i)} results from the inclusions
$
M_r \subseteq M_n
$
for all $n\ge r$.
Finally, the equivalence \textup{(iii)}$\Leftrightarrow$\textup{(ii)} follows directly from \Cref{local-global:positive}.
\end{proof}

We conclude this subsection by extending \Cref{local-global:positive,local-global:nonpositive} to symmetric cones. Replacing $\ZZ^{(\NN)}$ by $\RR^{(\NN)}$ and the monoidal operation $\mndcl$ by the conical operation $\cone$, the notions introduced in \Cref{sec:preliminaries}, such as \emph{$\Sym$-invariant chains} and \emph{stabilizing chains}, extend naturally to cones; see \cite{KLR22,LR23,L25} for further details. Recall that a cone $C\subseteq \RR^{(\NN)}$ is called \emph{pointed} if its \emph{lineality space}
\[
\lin(C)=\{\ub\in C\mid -\ub\in C\}
\]
is trivial. Classifications of non-pointed symmetric cones can be found in \cite[Section 5]{L25}.

\begin{proposition}\label{local-global:pointed cone}
Let $\C=(C_n)_{n\ge 1}$ be a stabilizing $\Sym$-invariant chain of cones with limit
$
C=\bigcup_{n\ge 1} C_n
$
and stability index $\ind(\C)=r$. Then the following statements are equivalent:
\begin{enumerate}
    \item $C_n$ is pointed for all $n\ge 1$ (respectively, non-pointed for all $n\ge r$);
    \item $C_r$ is pointed (respectively, non-pointed);
    \item $C$ is pointed (respectively, non-pointed).
\end{enumerate}
\end{proposition}

The proof of \Cref{local-global:pointed cone} proceeds by arguments analogous to those in the monoidal case and is therefore omitted.

\subsection{Local--global principles for normality, seminormality, and simplicity}

In view of \Cref{local-global:finite generation,local-global:positive,local-global:nonpositive}, it is natural to expect analogous local--global behavior for other fundamental properties of monoids, such as normality, seminormality, and simplicity. In particular, these results suggest that for a stabilizing $\Sym$-invariant chain $\M=(M_n)_{n\ge 1}$ with stability index $\ind(\M)=r$, the properties of being normal, seminormal, or simplicial propagate from $M_r$ to all $M_n$ with $n>r$. However, the following example shows that this expectation is false. 

\begin{example}\label{ex:not normal}
Let $\M=(M_n)_{n\ge 1}$ be a stabilizing $\Sym$-invariant chain of monoids with $\ind(\M)=2$ and 
$M_2=\mndcl\bigl(\Sym(2)\bigl((1,2),(1,1)\bigr)\bigr).$
We show that $M_2$ is simplicial and normal (and hence seminormal), whereas for all $n\ge 3$ the monoid $M_n$ is neither simplicial nor seminormal (and therefore not normal). Consequently, the limit monoid of this chain is neither normal nor seminormal.

First, observe that
\[
\cone(M_2)=\cone\bigl((1,2),(2,1)\bigr),
\]
which is a simplicial cone. 
Hence, $M_2$ is simplicial. Moreover,
\[
(1,0)=(2,1)-(1,1)\in \gp(M_2),
\]
and since $\gp(M_2)$ is $\Sym(2)$-invariant, it follows that $\gp(M_2)=\ZZ^2$. 
To show that $M_2$ is normal, it suffices to prove that
\[
\cone(M_2)\cap \gp(M_2)=\cone(M_2)\cap \ZZ^2\subseteq M_2,
\]
as the reverse inclusion is trivial. Let $\ub\in \cone(M_2)\cap \ZZ^2$. Then
\[
\ub=\lambda_1(1,2)+\lambda_2(2,1)
   =(\lambda_1+2\lambda_2,\,2\lambda_1+\lambda_2)
\]
for some $\lambda_1,\lambda_2\in \RR_{\ge 0}$. 

It follows that
\[
\lambda_1-\lambda_2=(2\lambda_1+\lambda_2)-(\lambda_1+2\lambda_2)\in \ZZ.
\]
Assume $\lambda_1-\lambda_2=a\in \ZZ_{\ge 0}$ (otherwise, interchange $\lambda_1$ and $\lambda_2$). Then
\[
\ub=(\lambda_2+a)(1,2)+\lambda_2(2,1)
   =3\lambda_2(1,1)+a(1,2).
\]
Since
\[
3\lambda_2=(\lambda_1+2\lambda_2)-(\lambda_1-\lambda_2)\in \ZZ
\]
and $\lambda_2\ge 0$, we have $3\lambda_2\in \ZZ_{\ge 0}$, and hence $\ub\in M_2$. This shows that $M_2$ is normal.

We next show that $M_n$ is not simplicial for all $n\ge3$. Indeed, $\cone(M_3)$ is minimally generated by six vectors, namely all permutations of $(1,2,0)$. Since this cone has dimension three, it is not simplicial, and hence neither is $M_3$. By \Cref{local-global:simplicial} below, it follows that $M_n$ is not simplicial for all $n>3$.

Finally, we show that $M_n$ is not seminormal for all $n\ge3$. By \Cref{local-global:normal-seminormal} below, it suffices to prove that $M_3$ is not seminormal. Note that $\gp(M_n)=\ZZ^n$ for all $n\ge 2$, since $\gp(M_2)=\ZZ^2$. Thus, in particular, $\epsilon_3\in \ZZ^3=\gp(M_3)$. We have
\begin{align*}
    2\epsilon_3&=(2,2,2)=(2,1,0)+(0,1,2)\in M_3,\\
    3\epsilon_3&=(3,3,3)=(2,1,0)+(0,2,1)+(1,0,2)\in M_3.
\end{align*}
Suppose, toward a contradiction, that
\[
\epsilon_3\in M_3=\mndcl\bigl(\Sym(3)\bigl((1,2),(1,1)\bigr)\bigr).
\]
Then
$3=s(\epsilon_3)=3a+2b$
for some
$a,b\in \ZZ_{\ge0},$
since $s((1,2))=3$ and $s((1,1))=2$. This forces $a=1$ and $b=0$. Hence $\epsilon_3=\sigma((1,2,0))$ for some $\sigma\in \Sym(3)$, a contradiction. Therefore, $\epsilon_3\notin M_3$, and $M_3$ is not seminormal.
\end{example} 

From \Cref{ex:not normal}, we see that the normality (respectively, seminormality, simplicity) of $M_r$ alone does not imply that the chain is eventually normal (respectively, eventually seminormal, eventually simplicial), 
and hence does not imply the normality (respectively, seminormality, simplicity) of the limit monoid. It is therefore natural to ask to what extent the normality, seminormality or simplicity of a global monoid is governed by the corresponding properties of the associated local monoids. Our next goal is to address this question for $\Sym$-equivariantly finitely generated monoids.

Let $\M=(M_n)_{n\ge 1}$ be a $\Sym$-invariant chain of monoids with limit monoid $M$. Assume that $M$ is $\Sym$-equivariantly finitely generated. By \Cref{local-global:finite generation}, the chain $\M$ stabilizes and is eventually finitely generated. Moreover, there exists an integer $q\ge \ind(\M)$ such that
\[
M_n=M\cap \ZZ^n=\mndcl(\Sym(n)(M_q))
\quad\text{for all }
n\ge q.
\]
The smallest integer $q$ with this property is denoted by $q(\M)$. The proof of \Cref{local-global:finite generation} further shows that $q(\M)$ is precisely the least integer such that
\[
M_n\cap \ZZ^k=M_k
\quad\text{for all }
n\ge k\ge q(\M).
\]

In answer to the aforementioned question, we obtain the following local–global principle for normality and seminormality.

 \begin{proposition}\label{local-global:normal-seminormal}
     Let $\M=(M_n)_{n\ge 1}$ be a $\Sym$-invariant chain of monoids with limit monoid $M=\bigcup_{n\ge 1}M_n$. Assume that $M$ is $\Sym$-equivariantly finitely generated. Then the following are equivalent:
    \begin{enumerate}
        \item 
        $\M$ is eventually normal (respectively, eventually seminormal);
        \item 
        $M_n$ is normal (respectively, seminormal) for all $n\ge q(\M)$;
        \item 
        $M$ is normal (respectively, seminormal).
    \end{enumerate}
 \end{proposition}

 \begin{proof}
    Since the argument for seminormality is entirely analogous, we present the proof only in the case of normality.
    \smallskip
    
    (i)$\Rightarrow$(ii): Choose $m>q(\M)$ such that $M_n$ is normal for all $n\ge m$. By induction, it suffices to show that $M_{m-1}$ is normal. Let $\ub\in \gp(M_{m-1})$ be such that $k\ub\in M_{m-1}$ for some integer $k\ge 1$. Since $M_{m-1}\subseteq M_m$, we have $\ub\in \gp(M_m)$ and $k\ub\in M_m$. By the normality of $M_m$, it follows that $\ub\in M_m$. 
    As $\ub\in \gp(M_{m-1})\subseteq \ZZ^{m-1}$ and $m-1\ge q(\M)$, we obtain
    \[
     \ub\in M_m\cap \ZZ^{m-1}=M_{m-1}.
    \]
    Thus $M_{m-1}$ is normal.
    \smallskip

    (ii)$\Rightarrow$(i): This implication is obvious.
    \smallskip
    
    (ii)$\Rightarrow$(iii): Assume $\ub \in \gp(M)$ such that $k\ub \in M$ for some $k\ge 1$. We need to show that $\ub \in M$. Since $M=\bigcup_{n\ge 1}M_n$, there exists $n \ge q(\M)$ such that $\ub \in \gp(M_n)$ and $k\ub \in M_n$. By assumption, $M_n$ is normal. Hence, $\ub \in M_n\subseteq M$.
    \smallskip

    (iii)$\Rightarrow$(ii): Fix $n\ge q(\M)$, and let $\ub \in \gp(M_n) \subseteq\ZZ^n$ be such that $k\ub \in M_n$ for some integer $k\ge 1$. Then $\ub \in \gp(M)$ and $k\ub \in M$. Since $M$ is normal, it follows that $\ub \in M$. Therefore, $\ub \in M\cap \ZZ^n=M_n$, which proves that $M_n$ is normal.
\end{proof}

We conclude this section with a discussion of the simplicity of symmetric monoids. The following result indicates that this condition is rather restrictive.

\begin{proposition}\label{local-global:simplicial}
     Let $\M=(M_n)_{n\ge 1}$ be a $\Sym$-invariant chain of monoids with limit monoid $M=\bigcup_{n\ge 1}M_n$ and stability index $\ind(\M)=r$. Assume that $M$ is $\Sym$-equivariantly finitely generated. Then the following are equivalent:
    \begin{enumerate}
        \item 
        $\M$ is eventually simplicial;
        \item 
        $M_n$ is simplicial for all $n\ge r$;
        \item 
        $M$ is simplicial;
        \item 
        Either $\cone(M)=\RR^{(\NN)}_{\ge0}$, or $\cone(M)=\RR^{(\NN)}_{\le0}$;
        \item 
        Either $\cone(M_n)=\RR^{n}_{\ge0}$ for all $n\ge r$, or $\cone(M_n)=\RR^{n}_{\le0}$ for all $n\ge r$.
    \end{enumerate}
 \end{proposition}

 To prove this result, we need some preparations. Let $C\subseteq\RR^{(\NN)}$ be a cone. An element $\ub\in C\setminus\{\nub\}$ is called an \emph{extreme generator} of $C$ if, whenever $\ub=\vb+\wb$ for some $\vb,\wb\in C$, one necessarily has $\vb,\wb\in\cone(\ub)=\RR_{\ge0}\ub.$ In this case, we say that $\cone(\ub)$ is an \emph{extreme ray} of $C$. 
 
 The set of all extreme rays of $C$ is denoted by $\rho(C)$. For any subset $A\subseteq\RR^{(\NN)}$, we define
 \[
 \R(A)\defas\{\cone(\ub)\mid \ub\in A\},
 \]
 which is the set of rays spanned by the elements of $A$.
 
 The following lemma is well-known in the finite-dimensional setting. We include a proof here for completeness.

 \begin{lemma}
     \label{lem:extreme}
     Let $C\subseteq\RR^{(\NN)}$ be a simplicial cone. Assume that $F$ and $G$ are generating sets for $C$. Assume further that $G$ is an independent set. Then the following hold:
     \begin{enumerate}
         \item 
         $C$ is pointed.
         \item 
         $\rho(C)=\R(G)\subseteq \R(F)$.
         \item 
         Let $H\subseteq C$ be a subset such that $\cone(\ub)\ne \cone(\vb)$ whenever $\ub,\vb\in H$ and $\ub\ne\vb$. If $\rho(C)=\R(H)$, then $H$ is an independent generating set for $C$.
     \end{enumerate}
 \end{lemma}

 \begin{proof}
     (i) Suppose there exists an element $\xb$ such that $\pm \xb\in C$. Then we can write
     \[
     \xb=\sum_{i=1}^ka_i \ub_i
     \quad\text{and}\quad
     -\xb=\sum_{i=1}^kb_i \ub_i,
     \]
     where $\ub_i\in G$ and $a_i,b_i\in \RR_{\ge0}$ for all $i\in[k]$. It follows that
     \[
     \nub=\sum_{i=1}^k(a_i+b_i) \ub_i.
     \]
     Since $G$ is independent, this forces $a_i=b_i=0$ for all $i\in[k]$. Hence, $\xb=\nub$, and therefore $C$ is pointed.

     (ii) We first show that $\rho(C)\subseteq \R(F)$. Let $\cone(\xb)\in\rho(C)$. Since $F$ generates $C$, we can write
     \[
     \xb=\sum_{i=1}^ka_i \ub_i,
     \quad\text{with }
     \ub_i\in F,\ a_i\in \RR_{\ge0}.
     \]
     As $\xb\ne\nub$, we may assume that $a_1>0$. Since $\cone(\xb)$ is an extreme ray, it follows that $a_1\ub_1\in\cone(\xb)$. Hence, $\cone(\xb)=\cone(\ub_1)\in \R(F)$, and therefore $\rho(C)\subseteq \R(F)$.
     
     Since $G$ is also a generating set for $C$, the above argument shows that $\rho(C)\subseteq \R(G)$. It remains to prove the reverse inclusion, which amounts to showing that every element $\ub\in G$ is an extreme generator of $C$. Suppose that $\ub=\vb+\wb$ for some $\vb,\wb\in C$. Writing
     \[
     \vb=a\ub+\sum_{i=1}^ka_i \ub_i
     \quad\text{and}\quad
     \wb=b\ub+\sum_{i=1}^kb_i \ub_i
     \]
     with $\ub_i\in G\setminus\{\ub\}$ and $a,b,a_i,b_i\in \RR_{\ge0}$ for all $i\in[k]$, we obtain
     \[
     \ub=(a+b)\ub+\sum_{i=1}^k(a_i+b_i) \ub_i.
     \]
     By the independence of $G$, we must have $a+b=1$ and $a_i=b_i=0$ for all $i\in[k]$. Hence, $\vb,\wb\in\cone(\ub)$, showing that $\ub$ is an extreme generator of $C$. Thus, $\rho(C)=\R(G)$.

     (iii) By part~(ii), we have $\R(H)=\R(G)$. Together with the assumption on $H$, this implies that for each $\ub\in G$ there exists a unique $\vb\in H$ such that $\ub=a\vb$ for some $a>0$. Since $G$ is an independent generating set for $C$, it follows immediately that $H$ is also an independent generating set for $C$.
 \end{proof}

When a simplicial cone $C\subseteq\RR^{(\NN)}$ is $\Sym$-invariant, it is natural to ask whether $C$ admits a $\Sym$-equivariant independent generating set. The next result addresses this question in the finitely generated case. We formulate the statement for an arbitrary subgroup $\Omega\subseteq\Sym$, since the case $\Omega=\Sym(n)$ for $n\in\NN$ will also be needed later.

 \begin{lemma}
     \label{cor:equi-independence}
     Let $C\subseteq\RR^{(\NN)}$ be a simplicial cone that is $\Omega$-invariant for some subgroup $\Omega\subseteq\Sym$. Assume further that $C$ is $\Omega$-equivariantly generated by a finite set $F$. Then there exists a subset $G\subseteq F$ such that $\Omega(G)$ is an independent generating set for $C$. 
 \end{lemma}
    
 \begin{proof}
     Suppose $F=\{\ub_1,\dots,\ub_m\}$. We construct $G$ inductively as follows. Initialize $G=\emptyset$. For $i=1,\dots,m$, append $\ub_i$ to $G$ if and only if $\ub_i$ is an extreme generator of $C$ and
     \[
     \cone(\ub_i)\neq\cone(\sigma(\ub_j))
     \quad\text{for all }
     \sigma\in\Omega \text{ and all }
     j<i.
     \]
     By construction, the resulting set $G$ is a maximal subset of $F$ with the following properties:
     \begin{enumerate}
         \item 
         every element of $G$ is an extreme generator of $C$;
         \item 
         if $\ub,\vb\in G$ are distinct elements, then $\cone(\ub)\ne\cone(\sigma(\vb))$ for all $\sigma\in\Omega$. 
     \end{enumerate}
     
     We now show that $\Omega(G)$ is an independent generating set for $C$. By \Cref{lem:extreme}(iii), it suffices to prove that $\rho(C)=\R(\Omega(G)).$ Since $C$ is $\Omega$-invariant, it is easily seen that $\Omega(G)$ consists of extreme generators of $C$. Hence, $\R(\Omega(G))\subseteq\rho(C).$ 
     
     Conversely, let $\cone(\ub)\in \rho(C)$. Since $\Omega(F)$ generates $C$, \Cref{lem:extreme}(ii) implies that $\rho(C)\subseteq \R(\Omega(F))$. Thus, $\cone(\ub)=\cone(\tau(\ub_i))$ for some $\tau\in\Omega$ and $i\in[m]$. By the construction of $G$, either $\ub_i\in G$, or there exist $\sigma\in\Omega$ and $\ub_j\in G$ with $j<i$ such that $\cone(\ub_i)=\cone(\sigma(\ub_j))$. In either case, we obtain $\cone(\ub)=\cone(\pi(\vb))$ for some $\pi\in\Omega$ and $\vb\in G$, showing that $\cone(\ub)\in\R(\Omega(G))$. This concludes the proof.
 \end{proof}

 \begin{lemma}
     \label{lem:symmetric-extreme}
     Assume that $C\subseteq\RR^{(\NN)}$ is a simplicial cone. If $C$ is $\Sym$-equivariantly finitely generated, then either $C=\RR^{(\NN)}_{\ge0}$ or $C=\RR^{(\NN)}_{\le0}$.
 \end{lemma}

 \begin{proof}
     By \Cref{cor:equi-independence}, there exists a finite subset $G\subseteq C$ such that $\Sym(G)$ is an independent generating set for $C$. We first claim that every element of $G$ has support size $1$. Indeed, choose $p\in\NN$ such that $G\subseteq\RR^p$. For $n\ge p$, set
     \[
     G_n=\Sym(n)(G)
     \quad\text{and}\quad
     C_n=\cone(G_n)\subseteq\RR^n.
     \]
     Since $G_n\subseteq\Sym(G)$, it is an independent set. Hence, $C_n$ is a simplicial cone for all $n\ge p$. Suppose that there exists $\ub\in G$ with $|\supp(\ub)|\ge 2$. Permuting the coordinates if necessary, we may assume that $\ub=(u_1,u_2,\dots)$ with $u_1,u_2\ne0$. 
     
     For $n>p$, let $\vb,\wb$ be the vectors obtained from $\ub$ by exchanging the $n$-th coordinate of $\ub$ (which is $0$) with its first and second coordinates, respectively. Then $\vb\ne\wb$ and
     $
     \vb,\wb\in G_n\setminus G_{n-1}.
     $
     Consequently,
     \[
     \dim\, C_n=|G_n|\ge |G_{n-1}|+2=\dim\, C_{n-1}+2
     \quad\text{for all } n>p.
     \]
    Iterating this inequality shows that $\dim C_n>n$ for $n\gg 0$, which is impossible since $C_n\subseteq\RR^n$ for all $n\ge p$. This contradiction proves that every element of $G$ has support size~$1$. Thus, each $\ub\in G$ is of the form $a_i\eb_i$ for some $a_i\neq 0$ and $i\in\NN$. Since $C$ is pointed by \Cref{lem:extreme}, all the $a_i$ must have the same sign. It follows that $C=\RR^{(\NN)}_{\ge 0}$ or $C=\RR^{(\NN)}_{\le 0}$, as claimed.
 \end{proof}

 We are now in a position to prove \Cref{local-global:simplicial}.

 \begin{proof}[Proof of \Cref{local-global:simplicial}]
     The implications (v)$\Rightarrow$(ii)$\Rightarrow$(i) and (v)$\Rightarrow$(iv) are trivial. Moreover, the equivalence (iii)$\Leftrightarrow$(iv) follows directly from \Cref{lem:symmetric-extreme}.

     (iv)$\Rightarrow$(v): We may assume that $\cone(M)=\RR^{(\NN)}_{\ge0}$, since the remaining case then follows by replacing $M$ with $-M$. Observe that 
     \[
     \cone(M_n)\subseteq\cone(M)\cap\RR^n=\RR^{n}_{\ge0}
     \quad\text{for all } n\ge 1.
     \]
     Evidently, $\cone(M)$ is the limit of the chain $(\cone(M_n))_{n\ge1}$, which has stability index $r$. By \cite[Corollary 5.4]{KLR22} (see also \cite[Lemma 5.1]{KLR22}), we obtain
     \[
     \cone(M_n)=\cone(M)\cap\RR^n=\RR^{n}_{\ge0}
     \quad\text{for all } n\ge r.
     \]

     (i)$\Rightarrow$(iii): Choose $p\ge r$ such that $\cone(M_n)$ is simplicial for all $n\ge p$. Then, in particular, $\cone(M_n)$ is finitely generated for all $n\ge p$. By \Cref{cor:equi-independence}, there exists a finite set $G_p$ such that $\Sym(p)(G_p)$ is an independent generating for $\cone(M_p)$. Since $p\ge r$, the set $\Sym(p+1)(G_p)$ generates $\cone(M_{p+1})$. Applying \Cref{cor:equi-independence} again, we obtain a subset $G_{p+1}\subseteq G_p$ such that $\Sym(p+1)(G_{p+1})$ is an independent generating set for $\cone(M_{p+1})$. Iterating this process yields a nested chain
     \[
     G_p\supseteq G_{p+1}\supseteq G_{p+2}\supseteq\cdots
     \]
     with the property that $\Sym(n)(G_{n})$ is an independent generating set for $\cone(M_{n})$ for all $n\ge p$. Since $G_p$ is finite, this chain stabilizes. Thus, there exist $m\ge p$ and a finite set $G$ such that $\Sym(n)(G)$ is an independent generating set for $\cone(M_{n})$ for all $n\ge m$. Consequently, $\Sym(G)=\bigcup_{n\ge m}\Sym(n)(G)$ is an independent generating set for $\cone(M)$. Hence, $\cone(M)$ is simplicial, and so $M$ is simplicial. This completes the proof.
 \end{proof}

\section{Future directions}\label{sec:concluding remarks}

This paper develops a framework for the study of monoids up to symmetry. In this final section, we briefly outline some directions for future research.

Dual cones of symmetric cones have been investigated in \cite{L25,LR23}, where equivariant analogues of the classical Minkowski--Weyl theorem are established for symmetric cones. It remains an open question whether similar results hold for symmetric monoids; see \cite[Problem~7.1]{L25}. Our forthcoming paper \cite{LRV} addresses this problem by developing a framework for studying duals of monoids in the presence of symmetry.

Another promising direction concerns algebras associated with monoids in the equivariant setting. Understanding their algebraic invariants is of particular interest. As a first step in this direction, we propose the following problem.

\begin{problem}
    Let $\M=(M_{n})_{n\geq 1}$ be a $\Sym$-invariant chain of monoids. Study the Hilbert series of the monoids in this chain.
\end{problem}

\section*{Acknowledgements}
The first author was supported by the Vietnam National Foundation for Science and Technology Development (NAFOSTED) under the grant number 101.04-2025.49. The second and third author were supported by the SPP 2458 “Combinatorial Synergies”, funded by the Deutsche Forschungsgemeinschaft (DFG, German Research Foundation) under the grant number GZ: RO 2504/6-1; AOBJ: 704082.

\end{document}